\documentclass[12pt,reqno]{amsart}

\usepackage{amsfonts, amsthm, amsmath}
\allowdisplaybreaks[4]

\usepackage{rotating}

\usepackage{tikz}

\usepackage{graphics}

\usepackage{amssymb}

\usepackage{amscd}

\usepackage[latin2]{inputenc}

\usepackage{t1enc}

\usepackage[mathscr]{eucal}

\usepackage{indentfirst}

\usepackage{graphicx}

\usepackage{graphics}

\usepackage{pict2e}

\usepackage{mathrsfs}

\usepackage{enumerate}
\usepackage[pagebackref]{hyperref}
\hypersetup{backref, pagebackref, colorlinks=true}
\usepackage{cite}
\usepackage{color}
\usepackage{epic}
\usepackage{hyperref} 
\numberwithin{equation}{section}
\theoremstyle{plain}

\newtheorem{theorem}{Theorem}[section]

\newtheorem{lemma}[theorem]{Lemma}

\theoremstyle{definition}

\newtheorem{remark}[theorem]{Remark}

\newtheorem{?}[theorem]{Problem}

\def\boxit#1{\leavevmode\hbox{\vrule\vtop{\vbox{\kern.33333pt\hrule
    \kern1pt\hbox{\kern1pt\vbox{#1}\kern1pt}}\kern1pt\hrule}\vrule}}

\usepackage{collectbox}



\newcommand{\f}[1]{\ifthenelse{\equal{#1}{1}}{(q;q)_\infty}{(q^{#1};q^{#1})_{\infty}}}

\begin{document}
\title[Congruences modulo powers of 3]{Congruences modulo powers of 3 for 2-color partition triples}

\author[D. Tang]{Dazhao Tang}

\address[Dazhao Tang]{College of Mathematics and Statistics, Chongqing University, Huxi Campus LD206, Chongqing 401331, P.R. China}
\email{dazhaotang@sina.com}

\date{\today}

\begin{abstract}
Let $p_{k,3}(n)$ enumerate the number of 2-color partition triples of $n$ where one of the colors appears only in parts that are multiples of $k$. In this paper, we prove several infinite families of congruences modulo powers of 3 for $p_{k,3}(n)$ with $k=1, 3$, and $9$. For example, for all integers $n\geq0$ and $\alpha\geq1$, we prove that
\begin{align*}
p_{3,3}\left(3^{\alpha}n+\dfrac{3^{\alpha}+1}{2}\right) &\equiv0\pmod{3^{\alpha+1}}
\end{align*}
and
\begin{align*}
p_{3,3}\left(3^{\alpha+1}n+\dfrac{5\times3^{\alpha}+1}{2}\right) &\equiv0\pmod{3^{\alpha+4}}.
\end{align*}
\end{abstract}



\maketitle

\section{Introduction}
The 2-color partition triple function is defined by \cite{CW2017}:
\begin{align*}
\sum_{n=0}^{\infty}p_{k,3}(n)q^{n}=\dfrac{1}{(q;q)_{\infty}^{3}(q^{k};q^{k})_{\infty}^{3}}.
\end{align*}

Here and throughout the paper, we adopt the following customary notation on partitions and $q$-series:
\begin{align*}
(a;q)_{\infty}=\prod_{n=0}^{\infty}(1-aq^{n}),\quad |q|<1.
\end{align*}

Partition-theoretically, $p_{k,3}(n)$ can be interpreted as the number of 2-color partition triples of $n$ where one of the colors appears only in parts that are multiples of $k$. For instance, $p_{2,3}(2)=12$, since $2=2_{1}=2_{2}=2_{3}=2_{4}=2_{5}=2_{6}=1_{1}+1_{1}=1_{2}+1_{2}=1_{3}+1_{3}\\=1_{1}+1_{2}=1_{1}+1_{3}=1_{2}+1_{3}$.

Quite recently, Wang and Chern \cite{CW2017} established some Ramanujan-type congruences modulo 5, 7, and 11 for $p_{3,3}(n)$ via modular forms and standard $q$-series techniques. Following their work and relating to Hirschhorn's recent work on 3-colored partitions \cite{Hir2017}, we will obtain some infinite families of congruences modulo any powers of 3 for $p_{k,3}(n)$ with $k=1, 3$, and $9$. The main results of this paper are stated as follows:
\begin{theorem}\label{thm1:power 3}
For all integers $n\geq0$ and $\alpha\geq1$,
\begin{align}
p_{1,3}\left(3^{2\alpha-1}n+\dfrac{3^{2\alpha-1}+1}{4}\right) &\equiv0\pmod{3^{2\alpha-1}},\label{case 1:odd}\\
p_{1,3}\left(3^{2\alpha+1}n+\dfrac{7\times3^{2\alpha}+1}{4}\right) &\equiv0\pmod{3^{2\alpha+3}},\label{case 2:3n+1}\\
p_{1,3}\left(3^{2\alpha+1}n+\dfrac{11\times3^{2\alpha}+1}{4}\right) &\equiv0\pmod{3^{2\alpha+4}},\label{case 2:3n+2}\\
p_{1,3}\left(3^{2\alpha}n+\dfrac{5\times3^{2\alpha-1}+1}{4}\right) &\equiv0\pmod{3^{2\alpha}},\label{case 1:3n+1}\\
p_{1,3}\left(3^{2\alpha}n+\dfrac{3^{2\alpha+1}+1}{4}\right) &\equiv0\pmod{3^{2\alpha+1}}.\label{case 1:even}
\end{align}
\end{theorem}

\begin{theorem}
For all integers $n\geq0$ and $\alpha\geq1$,
\begin{align}
p_{3,3}\left(3^{\alpha}n+\dfrac{3^{\alpha}+1}{2}\right) &\equiv0\pmod{3^{\alpha+1}},\label{case 3:3n+1}\\
p_{3,3}\left(3^{\alpha+1}n+\dfrac{5\times3^{\alpha}+1}{2}\right) &\equiv0\pmod{3^{\alpha+4}}.\label{case 3:3n+2}
\end{align}
\end{theorem}

\begin{theorem}\label{thm2:power 3}
For all integers $n\geq0$ and $\alpha\geq1$,
\begin{align}
p_{9,3}\left(3^{2\alpha-1}n+\dfrac{3^{2\alpha-1}+5}{4}\right) &\equiv0\pmod{3^{2\alpha}},\label{case 9:3n}\\
p_{9,3}\left(3^{2\alpha+1}n+\dfrac{7\times3^{2\alpha}+5}{4}\right) &\equiv0\pmod{3^{2\alpha+4}},\label{case 9-2:3n+1}\\
p_{9,3}\left(3^{2\alpha+1}n+\dfrac{11\times3^{2\alpha}+5}{4}\right) &\equiv0\pmod{3^{2\alpha+5}},\label{case 9-2:3n+2}\\
p_{9,3}\left(3^{2\alpha}n+\dfrac{5\times3^{2\alpha-1}+5}{4}\right) &\equiv0\pmod{3^{2\alpha+1}},\label{case 9-1:3n+1}\\
p_{9,3}\left(3^{2\alpha}n+\dfrac{3^{2\alpha+1}+5}{4}\right) &\equiv0\pmod{3^{2\alpha+2}}.\label{case 9:3n+2}
\end{align}
\end{theorem}

Of course, for any positive integer $k=3m$, we can also investigate arithmetic properties modulo powers of 3 for $p_{k,3}(n)$  case by case via following the same line of proving Theorems \ref{thm1:power 3}--\ref{thm2:power 3}. However, the calculations get more and more tedious.

The remainder of this paper is organized as follows. In Section \ref{lemmas section}, we present the background material on the $H$ operator and some necessary lemmas. In Section \ref{main results}, we provide the proofs of Theorems \ref{thm1:power 3}--\ref{thm2:power 3}. We end with some remarks to motivate further
study.

\section{Preliminary results}\label{lemmas section}
To obtain the main results of this paper, we first introduce some necessary notations and terminology on $q$-series.

For notational convenience, we denote that
\begin{align*}
E_{j} :=E(q^{j})=(q^{j};q^{j})_{\infty}.
\end{align*}

The following 3-dissection identity was proved by Hirschhorn \cite[Eq. (21.3.3)]{Hirb2017}.
\begin{lemma}\label{dissection lemma}
There holds
\begin{align*}
E_{1}^{3} &=A(q^{3})-3qE_{9}^{3}.
\end{align*}
where
\begin{align*}
A(q)=\dfrac{E_{2}^{6}E_{3}}{E_{1}^{2}E_{6}^{2}}+3q\dfrac{E_{1}^{2}E_{6}^{6}}{E_{2}^{2}E_{3}^{3}}.
\end{align*}
\end{lemma}

Following Hirschhron \cite{Hirb2017, Hir2017}, we introduce a ``huffing'' operator $H$, given by
\begin{align*}
H\left(\sum_{n=0}^{\infty}a_{n}q^{n}\right)=\sum_{n=0}^{\infty}a_{3n}q^{3n}.
\end{align*}

As in Hirschhorn \cite{Hir2017}, we define an infinite matrix $\{m(i,j)\}_{i,j\geq1}$ by
\begin{align}\label{3-power table}
\begin{pmatrix}
9 &0 &0 &0 &0 &0 &\cdots\\ 6 &243 &0 &0 &0 &0 &\cdots\\ 1 &243 &6561 &0 &0 &0 &\cdots\\ 0 &90 &8748 &177147 &0 &0 &\cdots\\ 0 &15 &4860 &295245 &4782969 &0 &\cdots\\ \vdots &\vdots &\vdots &\vdots &\vdots &\vdots &\ddots
\end{pmatrix}
\end{align}
and for $i\geq4$, $m(i,1)=0$, and for $j\geq2$,
\begin{align}
m(i,j) &=27m(i-1,j-1)+9m(i-2,j-1)+m(i-3,j-1).\label{recu formula}
\end{align}
By induction, it follows immediately from \eqref{recu formula} that
\begin{lemma}\label{zero vaules}
The coefficients $m(i,j)\neq0$ if and only if
\begin{align}
\left\lfloor\dfrac{i+2}{3}\right\rfloor\leq j\leq i.\label{bibounded ineq}
\end{align}
\end{lemma}

The following lemma is the main ingredient for our proof.
\begin{lemma}[Eq. (2.26), \cite{Hir2017}]\label{key lemma}
For any integer $i\geq1$,
\begin{align}
H\left(\dfrac{1}{\zeta^{i}}\right)=\sum_{j=1}^{\infty}\dfrac{m(i,j)}{T^{j}}=\sum_{j=1}^{i}\dfrac{m(i,j)}{T^{j}},\label{H operator}
\end{align}
where
\begin{align*}
\zeta=\dfrac{E_{1}^{3}}{qE_{9}^{3}}, \quad T=\dfrac{E_{3}^{12}}{q^{3}E_{9}^{12}}.
\end{align*}
\end{lemma}

\section{Congruences modulo any powers of 3}\label{main results}
\subsection{Congruences for $p_{1,3}(n)$ modulo powers of 3}
In this subsection, we will establish \eqref{case 1:odd}--\eqref{case 1:even}. To describe the generating functions for the sequences in \eqref{case 1:odd} and \eqref{case 1:even}, we define another infinite matrix of natural integers $\{a(j,k)\}_{j,k\geq1}$ by
\begin{enumerate}[1)]
\item $a(1,1)=6$, $a(1,2)=243$, and $a(1,k)=0$ for $k\geq3$.
\item For all integers $j\geq1$ and $k\geq1$,
\begin{align*}
a(j+1,k)=
\begin{cases}
\sum_{i=1}^{\infty}a(j,i)m(4i,i+k)\quad &\textrm{if}~j~\textrm{is~odd}, \cr \sum_{i=1}^{\infty}a(j,i)m(4i+2,i+k)\quad &\textrm{if}~j~\textrm{is~even}.
\end{cases}
\end{align*}
\end{enumerate}
Notice that by Lemma \ref{zero vaules}, the summation in $2)$ is in fact finite.

To obtain \eqref{case 1:odd}--\eqref{case 1:even}, we need the following key theorem and lemmas.
\begin{theorem}
For any positive integer $j$,
\begin{align}
\sum_{n=0}^{\infty}p_{1,3}\left(3^{2j-1}n+\dfrac{3^{2j-1}+1}{4}\right)q^{n} &=\sum_{l=1}^{\infty}a(2j-1,l)q^{l-1}\dfrac{E_{3}^{12l-6}}{E_{1}^{12l}},\label{gf1:odd}\\
\sum_{n=0}^{\infty}p_{1,3}\left(3^{2j}n+\dfrac{3^{2j+1}+1}{4}\right)q^{n} &=\sum_{l=1}^{\infty}a(2j,l)q^{l-1}\dfrac{E_{3}^{12l}}{E_{1}^{12l+6}}.\label{gf1:even}
\end{align}
\end{theorem}

\begin{proof}
By the definition of $\zeta$, we see that
\begin{align}
\sum_{n=0}^{\infty}p_{1,3}(n)q^{n}=\dfrac{1}{E_{1}^{6}}=\zeta^{-2}\dfrac{1}{q^{2}E_{9}^{6}}.\label{gf:p_{1,3}}
\end{align}

Applying the $H$ operator on \eqref{gf:p_{1,3}} and picking out those terms of the form $q^{3n+1}$, after simplification, we find that
\begin{align}
\sum_{n=0}^{\infty}p_{1,3}(3n+1)q^{n}=6\dfrac{E_{3}^{6}}{E_{1}^{12}}+243q\dfrac{E_{3}^{18}}{E_{1}^{24}}.\label{initial case:1}
\end{align}

Next, we proceed by induction on $j$. According to \eqref{initial case:1}, we know that \eqref{gf1:odd} is true for $j=1$. Assume \eqref{gf1:odd} holds for some positive integer $j\geq1$. We can rewrite it as
\begin{align*}
\sum_{n=0}^{\infty}p_{1,3}\left(3^{2j-1}n+\dfrac{3^{2j-1}+1}{4}\right)q^{n}=\dfrac{1}{qE_{3}^{6}}\sum_{l=1}^{\infty}a(2j-1,l)T^{l}\zeta^{-4l}.
\end{align*}
Taking out those terms of the form $q^{3n+2}$ and applying Lemma \ref{key lemma}, we have
\begin{align}
 \sum_{n=0}^{\infty}p_{1,3}\left(3^{2j-1}(3n+2)+\dfrac{3^{2j-1}+1}{4}\right)q^{3n+2}\notag =&\dfrac{1}{qE_{3}^{6}}\sum_{l=1}^{\infty}a(2j-1,l)T^{l}H\left(\zeta^{-4l}\right)\notag\\
 =&\dfrac{1}{qE_{3}^{6}}\sum_{l=1}^{\infty}a(2j-1,l)T^{l}\left(\sum_{k=1}^{\infty}m(4l,k)T^{-k}\right).\label{term:3n+2}
\end{align}
According to Lemma \ref{zero vaules}, we obtain that $m(4l,k)\neq0$ when $\left\lfloor\dfrac{4l+2}{3}\right\rfloor\leq k$, thus we can suppose $k\geq l+1$. Now \eqref{term:3n+2} implies
\begin{align*}
 &\sum_{n=0}^{\infty}p_{1,3}\left(3^{2j+1}n+\dfrac{3^{2j+1}+1}{4}\right)q^{n}\\
 =&\dfrac{1}{qE_{1}^{6}}\sum_{l=1}^{\infty}\sum_{k=l+1}^{\infty}a(2j-1,l)m(4l,k)\left(\dfrac{qE_{3}^{12}}{E_{1}^{12}}\right)^{k-l} \quad(\textrm{replace}~k~\textrm{by}~k+l)\\
 =&\dfrac{1}{qE_{1}^{6}}\sum_{k=1}^{\infty}\sum_{l=1}^{\infty}a(2j-1,l)m(4l,k+l)\left(\dfrac{qE_{3}^{12}}{E_{1}^{12}}\right)^{k}\\
 =&\sum_{k=1}^{\infty}a(2j,k)q^{k-1}\dfrac{E_{3}^{12k}}{E_{1}^{12k+6}}.
\end{align*}
This implies that \eqref{gf1:even} holds for $j$. Similarly, we rewrite \eqref{gf1:even} as follows:
\begin{align*}
\sum_{n=0}^{\infty}p_{1,3}\left(3^{2j}n+\dfrac{3^{2j+1}+1}{4}\right)q^{n}=\dfrac{1}{q^{3}E_{9}^{6}}\sum_{l=1}^{\infty}a(2j,l)T^{l}\zeta^{-(4l+2)}.
\end{align*}
Picking out those terms of the form $q^{3n}$ and according to Lemma \ref{key lemma}, we obtain
\begin{align}
 \sum_{n=0}^{\infty}p_{1,3}\left(3^{2j}(3n)+\dfrac{3^{2j+1}+1}{4}\right)q^{3n} &=\dfrac{1}{q^{3}E_{9}^{6}}\sum_{l=1}^{\infty}a(2j,l)T^{l}H\left(\zeta^{-(4l+2)}\right)\nonumber\\
 &=\dfrac{1}{q^{3}E_{9}^{6}}\sum_{l=1}^{\infty}a(2j,l)T^{l}\left(\sum_{k=1}^{\infty}m(4l+2,k)T^{-k}\right).\label{term:3n}
\end{align}
By Lemma \ref{zero vaules}, we can see that $m(4l+2,k)\neq0$ when $\left\lfloor\dfrac{4l+4}{3}\right\rfloor\leq k$, thus we will suppose $k\geq l+1$. Now \eqref{term:3n} implies
\begin{align*}
 &\sum_{n=0}^{\infty}p_{1,3}\left(3^{2j+1}n+\dfrac{3^{2j+1}+1}{4}\right)q^{n}\\
 =&\dfrac{1}{qE_{3}^{6}}\sum_{l=1}^{\infty}\sum_{k=l+1}^{\infty}a(2j,l)m(4l+2,k)\left(\dfrac{qE_{3}^{12}}{E_{1}^{12}}\right)^{k-l} \quad(\textrm{replace}~k~\textrm{by}~k+l)\\
 =&\dfrac{1}{qE_{3}^{6}}\sum_{k=1}^{\infty}\sum_{l=1}^{\infty}a(2j,l)m(4l+2,k+l)\left(\dfrac{qE_{3}^{12}}{E_{1}^{12}}\right)^{k}\\
 =&\sum_{k=1}^{\infty}a(2j+1,k)q^{k-1}\dfrac{E_{3}^{12k-6}}{E_{1}^{12k}}.
\end{align*}
This implies that \eqref{gf1:odd} holds for $j+1$. Hence we finish the proof by induction.
\end{proof}

For any positive integer $n$, let $\nu_{3}(n)$ count the highest power of 3 that divides $n$. For convention, we denote $\nu_{3}(0)=\infty$. To prove \eqref{case 1:odd}--\eqref{case 1:even}, we need the following lemma to estimate the 3-adic powers of $a(j,k)$.
\begin{lemma}\label{estimate lemma}
For any positive integers $i\geq1$ and $j\geq1$,
\begin{align}
\nu_{3}(m(i,j)) &\geq\left\lfloor\dfrac{9j-3i-1}{2}\right\rfloor.\label{condition 1}
\end{align}
\end{lemma}
\begin{proof}
Eq. \eqref{condition 1} follows easily from \eqref{recu formula} and induction on $i$, $j$.
\end{proof}

\begin{remark}
Two remarks on Lemma \ref{estimate lemma} are in order. First, Hirschhorn \cite[Eq. (4.1)]{Hir2017} obtained $\nu_{3}(m(i,j))\geq\dfrac{9j-3i-3}{2}$. From this perspective, Eq. \eqref{condition 1} can be viewed as an improvement on the lower bound of $\nu_{3}(m(i,j))$. On the other hand, in view of \eqref{3-power table}, we would like to emphasize that the lower bound given by Eq. \eqref{condition 1} is best possible.
\end{remark}

\begin{lemma}\label{impor lemma}
For any positive integers $j\geq1$ and $k\geq1$,
\begin{align}
\nu_{3}(a(2j-1,k)) &\geq2j-1+\delta_{k,1}+\left\lfloor\dfrac{9k-10}{2}\right\rfloor,\label{odd estimate}\\
\nu_{3}(a(2j,k)) &\geq 2j+1+\delta_{k,1}+\left\lfloor\dfrac{9k-10}{2}\right\rfloor,\label{even estimate}
\end{align}
where $\delta_{k,l}$ is the Kronecker delta function, it equals 1 when $k=l$ and 0 otherwise.
\end{lemma}
\begin{proof}
It is easy to see that \eqref{odd estimate} holds for $j=1$. Assume \eqref{odd estimate} is true for some $j\geq1$. Then we need to consider the following two cases:
\begin{enumerate}[1)]
\item $k=1$. Thanks to \eqref{bibounded ineq}, we find that $m(4i,i+1)=0$ if $i\geq4$. With the aid of \eqref{recu formula}, we find that
\begin{align*}
m(4,2)=90,\quad m(8,3)=24,\quad m(12,4)=1.
\end{align*}
Hence we obtain
\begin{align*}
\nu_{3}(a(2j,1)) &=\nu_{3}\left(\sum_{i=1}^{\infty}a(2j-1,i)m(4i,i+1)\right)\\
 &=\nu_{3}(90a(2j-1,1)+24a(2j-1,2)+a(2j-1,3))\\
 &\geq\min\left\{2+\nu_{3}(a(2j-1,1)),1+\nu_{3}(a(2j-1,2)),\nu_{3}(a(2j-1,3))\right\}\\
 &\geq\min\{2j+1,2j+4,2j+7\}\\
 &=2j+1.
\end{align*}
This asserts \eqref{even estimate} holds for $k=1$.

\item $k>1$. According to Lemma \ref{estimate lemma}, we get
\begin{align*}
\nu_{3}(a(2j,k)) &=\nu_{3}\left(\sum_{i=1}^{\infty}a(2j-1,i)m(4i,k+i)\right)\\
 &\geq\min_{i\geq1}\left(\nu_{3}(a(2j-1,i))+\nu_{3}(m(4i,k+i))\right)\\
 &\geq\min_{i\geq1}\left(2j-1+\delta_{i,1}+\left\lfloor\dfrac{9i-10}{2}\right\rfloor+\left\lfloor\dfrac{9k-3i-1}{2}\right\rfloor\right)\\
 &\geq2j+2+\left\lfloor\dfrac{9k-10}{2}\right\rfloor.
\end{align*}
\end{enumerate}
These two cases imply \eqref{even estimate} holds for $j$.

Similarly, in view of \eqref{bibounded ineq}, we know that $m(4i+2,i+1)=0$ if $i\geq2$. By computation, we find that $m(6,2)=1$. Hence
\begin{align*}
a(2j+1,1)=\sum_{i=1}^{\infty}a(2j,i)m(4i+2,i+1)=a(2j,1).
\end{align*}
One readily obtain
\begin{align*}
\nu_{3}(a(2j+1,1))=\nu_{3}(a(2j,1))\geq2j+1.
\end{align*}

For $k>1$, we can see that
\begin{align*}
\nu_{3}(a(2j+1,k)) &=\nu_{3}\left(\sum_{i=1}^{\infty}a(2j,i)m(4i+2,k+i)\right)\\
 &\geq\min_{i\geq1}\left(\nu_{3}(a(2j,i))+\nu_{3}(m(4i+2,k+i))\right)\\
 &\geq\min_{i\geq1}\left(2j+1+\delta_{i,1}+\left\lfloor\dfrac{9i-10}{2}\right\rfloor+\left\lfloor\dfrac{9k-3i-7}{2}\right\rfloor\right)\\
 &\geq2j+1+\left\lfloor\dfrac{9k-10}{2}\right\rfloor.
\end{align*}
Therefore \eqref{odd estimate} is true for $j+1$ and therefore holds for any positive integer $j$. The proof is completed by induction.
\end{proof}

The congruence \eqref{case 1:odd} follows from \eqref{gf1:odd} together with \eqref{odd estimate}, and the congruence \eqref{case 1:even} follows from \eqref{gf1:even} together with \eqref{even estimate}.

Notice that
\begin{align*}
\nu_{3}(a(2j-1,k)) &\geq2j-1+\left\lfloor\dfrac{9k-10}{2}\right\rfloor\geq2j+3
\end{align*}
for $k\geq2$.

It follows that, modulo $3^{2j+3}$,
\begin{align*}
\sum_{n=0}^{\infty}p_{1,3}\left(3^{2j-1}n+\dfrac{3^{2j-1}+1}{4}\right)q^{n} &\equiv a(2j-1,1)\dfrac{E_{3}^{6}}{E_{1}^{12}}.
\end{align*}

With the help of Lemma \ref{dissection lemma} and \cite[Eq. (5.4)]{Hir2017}, we can see that, modulo $3^{2j+3}$,
\begin{align*}
 &\sum_{n=0}^{\infty}p_{1,3}\left(3^{2j-1}n +\dfrac{3^{2j-1}+1}{4}\right) q^{n}\\
 \equiv&a(2j-1,1)\dfrac{E_{9}^{12}}{E_{3}^{42}}\bigg(A(q^{3})^{8}+12qA(q^{3})^{7}E_{9}^{3}+90q^{2}A(q^{3})^{6}E_{9}^{6}+432q^{3}A(q^{3})^{5}E_{9}^{9}\\
 &\quad+1539A(q^{3})^{4}E_{9}^{12}+3888q^{5}A(q^{3})^{3}E_{9}^{15}+7290q^{6}A(q^{3})^{2}E_{9}^{18}\\
 &\quad+8748q^{7}A(q^{3})E_{9}^{21}+6561q^{8}E_{9}^{24}\bigg).
\end{align*}

If we extract those terms of the form $q^{3n+1}$, after manipulations, we find that, modulo $3^{2j+3}$,
\begin{align*}
 &\sum_{n=0}^{\infty}p_{1,3}\left(3^{2j-1}(3n+1)+\dfrac{3^{2j-1}+1}{4}\right)q^{n}=\sum_{n=0}^{\infty}p_{1,3}\left(3^{2j}+\dfrac{5\times3^{2j-1}+1}{4}\right)q^{n}\\ \equiv&3a(2j-1,1)\dfrac{E_{3}^{12}}{E_{1}^{42}}\left(4A(q)^{7}E_{3}^{3}+513qA(q)^{4}E_{3}^{12}+2916q^{2}A(q)E_{3}^{21}\right).
\end{align*}

Since $\nu_{3}(a(2j-1,1))\geq2j-1$, this establishes \eqref{case 1:3n+1}.

Similarly,
\begin{align*}
\nu_{3}(a(2j,k)) \geq2j+1+\left\lfloor\dfrac{9k-10}{2}\right\rfloor\geq2j+5
\end{align*}
for $k\geq2$.

It follows that, modulo $3^{2j+5}$,
\begin{align*}
\sum_{n=0}^{\infty}p_{1,3}\left(3^{2j}n+\dfrac{3^{2j+1}+1}{4}\right)q^{n} &\equiv a(2j,1)\dfrac{E_{3}^{12}}{E_{1}^{18}}.
\end{align*}

In view of Lemma \ref{dissection lemma}, we obtain, modulo $3^{2j+5}$, ($\omega$ is a cube root of unity different from one)
\begin{align*}
 &\sum_{n=0}^{\infty}p_{1,3}\left(3^{2j}n+\dfrac{3^{2j+1}+1}{4}\right)q^{n}\\
 \equiv& a(2j,1)E_{3}^{12}\dfrac{\left(E(\omega q)^{3}E(\omega^{2}q)^{3}\right)^{6}}{\left(E(q)^{3}E(\omega q)^{3}E(\omega^{2}q)^{3}\right)^{6}}\\
 =&a(2j,1)E_{3}^{12}\dfrac{\left(A(q^{3})^{2}+3qA(q^{3})E_{9}^{3}+9q^{2}E_{9}^{6}\right)^{6}}{\left(\dfrac{E_{3}^{4}}{E_{9}}\right)^{18}}\\
 =&a(2j,1)\dfrac{E_{9}^{18}}{E_{3}^{60}}\bigg(A(q^{3})^{12}+18qA(q^{3})^{11}E_{9}^{3}+189q^{2}A(q^{3})^{10}E_{9}^{6}+1350q^{3}A(q^{3})^{9}E_{9}^{9}+\\
 &\quad+7290q^{4}A(q^{3})^{8}E_{9}^{12}+30618q^{5}A(q^{3})^{7}E_{9}^{15}+102789q^{6}A(q^{3})^{6}E_{9}^{18}\\
 &\quad+275562q^{7}A(q^{3})^{5}E_{9}^{21}+590490q^{8}A(q^{3})^{4}E_{9}^{24}+984150q^{9}A(q^{3})^{3}E_{9}^{27}\\
 &\quad+1240029q^{10}A(q^{3})^{2}E_{9}^{30}+1062882q^{11}A(q^{3})E_{9}^{33}+531441q^{12}E_{9}^{36}\bigg).
\end{align*}

Picking out those terms of the form $q^{3n+1}$ and $q^{3n+2}$, after simplification, we find that, modulo $3^{2j+5}$,
\begin{align*}
 &\sum_{n=0}^{\infty}p_{1,3}\left(3^{2j}(3n+1)+\dfrac{3^{2j+1}+1}{4}\right)q^{n}=\sum_{n=0}^{\infty}p_{1,3}\left(3^{2j+1}n+\dfrac{7\times3^{2j}+1}{4}\right)q^{n}\\
 \equiv&9a(2j,1)\dfrac{E_{3}^{18}}{E_{1}^{60}}\left(2A(q)^{11}E_{3}^{3}+810qA(q)^{8}E_{3}^{12}+30618q^{2}A(q)^{5}E_{3}^{21}+137781q^{3}A(q)^{2}E_{3}^{30}\right)
\end{align*}
and
\begin{align*}
 &\sum_{n=0}^{\infty}p_{1,3}\left(3^{2j}(3n+2)+\dfrac{3^{2j+1}+1}{4}\right)q^{n}=\sum_{n=0}^{\infty}p_{1,3}\left(3^{2j+1}n+\dfrac{11\times3^{2j}+1}{4}\right)q^{n}\\
 \equiv&27a(2j,1)\dfrac{E_{3}^{18}}{E_{1}^{60}}\left(7A(q)^{10}E_{3}^{6}+1134qA(q)^{7}E_{3}^{15}+21870q^{2}A(q)^{4}E_{3}^{24}+39366q^{3}A(q)E_{3}^{33}\right).
\end{align*}

Notice that $\nu_{3}(a(2j,1))\geq2j+1$, this proves \eqref{case 2:3n+1} and \eqref{case 2:3n+2}.

\subsection{Congruences for $p_{3,3}(n)$ modulo powers of 3}
Now, we apply the same method to investigate the arithmetic properties for $p_{3,3}(n)$.  Define
\begin{enumerate}[1)]
\item $b(1,1)=9$, and $b(1,k)=0$ for $k\geq2$.
\item For any integers $j\geq1$ and $k\geq1$,
\begin{align*}
b(j+1,k)=\sum_{i=1}^{\infty}b(j,i)m(4i+1,i+k).
\end{align*}
\end{enumerate}

\begin{theorem}
For any positive integer $j$,
\begin{align}
\sum_{n=0}^{\infty}p_{3,3}\left(3^{j}n+\dfrac{3^{j}+1}{2}\right)q^{n} &=\sum_{l=1}^{\infty}b(j,l)q^{l-1}\dfrac{E_{3}^{12l-3}}{E_{1}^{12l+3}}.\label{gf3:odd}
\end{align}
\end{theorem}
\begin{proof}
By the similar reasons as above,
\begin{align*}
\sum_{n=0}^{\infty}p_{3,3}(n)q^{n}=\dfrac{1}{E_{1}^{3}E_{3}^{3}}=\zeta^{-1}\dfrac{1}{qE_{3}^{3}E_{9}^{3}}.
\end{align*}
Taking out those terms of the form $q^{3n+2}$ and invoking \eqref{H operator}, after simplification, we can see that
\begin{align}
\sum_{n=0}^{\infty}p_{3,3}(3n+2)q^{n}=9\dfrac{E_{3}^{9}}{E_{1}^{15}}.\label{initial case:2}
\end{align}

Next, we proceed by induction on $j$. By \eqref{initial case:2}, we know that \eqref{gf3:odd} holds for $j=1$. Suppose \eqref{gf3:odd} is true for some $j\geq1$. Notice that
\begin{align*}
\sum_{n=0}^{\infty}p_{3,3}\left(3^{j}n+\dfrac{3^{j}+1}{2}\right)q^{n} &=\dfrac{1}{q^{2}E_{3}^{3}E_{9}^{3}}\sum_{l=1}^{\infty}b(j,l)T^{l}\zeta^{-(4l+1)}.
\end{align*}
Taking out those terms of the form $q^{3n+1}$ and according to Lemma \ref{key lemma}, we obtain
\begin{align}
\sum_{n=0}^{\infty}p_{3,3}\left(3^{j}(3n+1)+\dfrac{3^{j}+1}{2}\right)q^{3n} &=\dfrac{1}{q^{3}E_{3}^{3}E_{9}^{3}}\sum_{l=1}^{\infty}b(j,l)T^{l}\left(\sum_{k=1}^{\infty}m(4l+1,k)T^{-k}\right).\label{term3:3n+1}
\end{align}

By Lemma \ref{zero vaules}, we obtain that $m(4l+1,k)\neq0$ when $\left\lfloor\dfrac{4l+3}{3}\right\rfloor\leq k$, thus we can suppose $k\geq l+1$. Now \eqref{term3:3n+1} implies
\begin{align*}
 &\sum_{n=0}^{\infty}p_{3,3}\left(3^{j+1}+\dfrac{3^{j+1}+1}{2}\right)q^{n}\\
 =&\dfrac{1}{qE_{1}^{3}E_{3}^{3}}\sum_{l=1}^{\infty}\sum_{k=l+1}^{\infty}b(j,l)m(4l+1,k)\left(\dfrac{qE_{3}^{12}}{E_{1}^{12}}\right)^{k-l} \quad(\textrm{replace}~k~\textrm{by}~k+l)\\
 =&\dfrac{1}{qE_{1}^{3}E_{3}^{3}}\sum_{k=1}^{\infty}\sum_{l=1}^{\infty}b(j,l)m(4l+1,k+l)\left(\dfrac{qE_{3}^{12}}{E_{1}^{12}}\right)^{k}\\
 =&\sum_{k=1}^{\infty}b(j+1,k)q^{k-1}\dfrac{E_{3}^{12k-3}}{E_{1}^{12k+3}}.
\end{align*}
This proves that \eqref{gf3:odd} is true for $j+1$. This finishes the proof by induction.
\end{proof}

\begin{lemma}
For any positive integer $j\geq1$, we have
\begin{align}
\nu_{3}(b(j,k)) &\geq j+1+\left\lfloor\dfrac{9k-9}{2}\right\rfloor.\label{estimate case 3}
\end{align}
\end{lemma}
\begin{proof}
It is easy to check that \eqref{estimate case 3} holds for $j=1$. Assume \eqref{estimate case 3} is true for some $j\geq1$. By Lemma \ref{estimate lemma}, we find that
\begin{align*}
\nu_{3}(b(j+1,k)) &=\nu_{3}\left(\sum_{i=1}^{\infty}b(j,i)m(4i+1,i+k)\right)\\
 &\geq\min_{i\geq1}\nu_{3}\left(b(j,i)m(4i+1,i+k)\right)\\
 &\geq\min_{i\geq1}\left(j+1+\left\lfloor\dfrac{9i-9}{2}\right\rfloor+\left\lfloor\dfrac{9k-3i-4}{2}\right\rfloor\right)\\
 &\geq j+2+\left\lfloor\dfrac{9k-9}{2}\right\rfloor.
\end{align*}
This proves that \eqref{estimate case 3} is true for $j+1$. This ends the proof by induction.
\end{proof}

The congruence \eqref{case 3:3n+1} follows from \eqref{gf3:odd} together with \eqref{estimate case 3}.

It is easy to see that
\begin{align*}
\nu_{3}(b(j,k))\geq j+1+\left\lfloor\dfrac{9k-9}{2}\right\rfloor\geq j+5
\end{align*}
for $k\geq2$.

Similarly, modulo $3^{j+5}$, we obtain
\begin{align*}
 &\sum_{n=0}^{\infty}p_{3,3}\left(3^{j}n+\dfrac{3^{j}+1}{2}\right)q^{n} \equiv b(j,1)\dfrac{E_{3}^{9}}{E_{1}^{15}}\\
 =&b(j,1)\dfrac{E_{3}^{15}}{E_{1}^{51}}\bigg(A(q^{3})^{10}+15qA(q^{3})^{9}E_{9}^{3}+135q^{2}A(q^{3})^{8}E_{9}^{6}+810q^{3}A(q^{3})^{7}E_{9}^{9}\\
 &\quad+3645q^{4}A(q^{3})^{6}E_{9}^{12}+12393q^{5}A(q^{3})^{5}E_{9}^{15}+32805q^{6}A(q^{3})^{4}E_{9}^{18}\\
 &\quad+65610q^{7}A(q^{3})^{3}E_{9}^{21}+98415q^{8}A(q^{3})^{2}E_{9}^{24}+98415q^{9}A(q^{3})E_{9}^{27}\\
 &\quad+59049q^{10}E_{9}^{30}\bigg).
\end{align*}

Taking out these terms of the form $q^{3n+2}$, after simplification, we find that, modulo $3^{j+5}$,
\begin{align*}
 &\sum_{n=0}^{\infty}p_{3,3}\left(3^{j}(3n+2)+\dfrac{3^{j}+1}{2}\right)q^{n}=\sum_{n=0}^{\infty}p_{3,3}\left(3^{j+1}n+\dfrac{5\times3^{j}+1}{2}\right)q^{n}\\
 \equiv&27b(j,1)\dfrac{E_{3}^{15}}{E_{1}^{51}}\left(5A(q)^{8}E_{3}^{6}+459qA(q)^{5}E_{3}^{15}+3645q^{2}A(q)^{2}E_{3}^{24}\right).
\end{align*}

Since $\nu_{3}(b(j,1))\geq j+1$, this establishes \eqref{case 3:3n+2}.

\subsection{Congruences for $p_{9,3}(n)$ modulo powers of 3}
We present here the main results and omit their proofs because this case is similar to the case $k=1$. Define
\begin{enumerate}[1)]
\item $c(1,1)=9$, and $c(1,k)=0$ for $k\geq2$.
\item For any integers $j\geq1$ and $k\geq1$,
\begin{align*}
c(j+1,k)=
\begin{cases}
\sum_{i=1}^{\infty}c(j,i)m(4i,i+k)\quad &\textrm{if}~j~\textrm{is~odd}, \cr \sum_{i=1}^{\infty}c(j,i)m(4i+2,i+k)\quad &\textrm{if}~j~\textrm{is~even}.
\end{cases}
\end{align*}
\end{enumerate}

\begin{theorem}
For any positive integer $j$,
\begin{align}
\sum_{n=0}^{\infty}p_{9,3}\left(3^{2j-1}n+\dfrac{3^{2j-1}+5}{4}\right)q^{n} &=\sum_{l=1}^{\infty}c(2j-1,l)q^{l-1}\dfrac{E_{3}^{12l-6}}{E_{1}^{12l}},\label{gf9:odd}\\
\sum_{n=0}^{\infty}p_{9,3}\left(3^{2j}n+\dfrac{3^{2j+1}+5}{4}\right)q^{n} &=\sum_{l=1}^{\infty}c(2j,l)q^{l-1}\dfrac{E_{3}^{12l}}{E_{1}^{12l+6}}.\label{gf9:even}
\end{align}
\end{theorem}

\begin{lemma}
For any positive integers $j\geq1$ and $k\geq1$,
\begin{align}
\nu_{3}(c(2j-1,k)) &\geq2j+\delta_{k,1}+\left\lfloor\dfrac{9k-10}{2}\right\rfloor,\label{odd estimate:9}\\
\nu_{3}(c(2j,k)) &\geq2j+2+\delta_{k,1}+\left\lfloor\dfrac{9k-10}{2}\right\rfloor.\label{even estimate:9}
\end{align}
\end{lemma}

Notice that \eqref{gf9:odd} and \eqref{odd estimate:9} imply the congruence \eqref{case 9:3n}, \eqref{gf9:even} and \eqref{even estimate:9} imply the congruence \eqref{case 9:3n+2}, respectively.

The proofs of \eqref{case 9-2:3n+1}--\eqref{case 9-1:3n+1} are similar to \eqref{case 2:3n+1}--\eqref{case 1:3n+1}.

\section{Final remarks}
According to Lemma \ref{dissection lemma}, we find that
\begin{align*}
\sum_{n=0}^{\infty}p_{3,3}(n)q^{n} &=\dfrac{1}{E_{1}^{3}E_{3}^{3}}=\dfrac{1}{E_{3}^{3}}\dfrac{E(\omega q)^{3}E(\omega^{2}q)^{3}}
{E(q)^{3}E(\omega q)^{3}E(\omega^{2}q)^{3}}\\
 &=\dfrac{E_{9}^{3}}{E_{3}^{15}}\left(A(q^{3})^{2}+3qA(q^{3})E_{9}^{3}+9q^{2}E_{9}^{6}\right)
\end{align*}
and
\begin{align*}
\sum_{n=0}^{\infty}p_{3,3}(3n)q^{n} &=\dfrac{E_{3}^{3}A(q)^{2}}{E_{1}^{15}},\\
\sum_{n=0}^{\infty}p_{3,3}(3n+1)q^{n} &=3\dfrac{E_{3}^{6}A(q)}{E_{1}^{15}}.
\end{align*}
Therefore, we have
\begin{align}
\sum_{n=0}^{\infty}p_{3,3}(3n)q^{n}\sum_{n=0}^{\infty}p_{3,3}(3n+2)q^{n}=\left(\sum_{n=0}^{\infty}p_{3,3}(3n+1)q^{n}\right)^{2}.\label{p3,3-beau relation}
\end{align}

Similarly, for $p_{9,3}(n)$, we also deduce that
\begin{align}
\sum_{n=0}^{\infty}p_{9,3}(3n)q^{n}\sum_{n=0}^{\infty}p_{9,3}(3n+2)q^{n}=\left(\sum_{n=0}^{\infty}p_{9,3}(3n+1)q^{n}\right)^{2}.\label{p9,3-beau relation}
\end{align}

Interestingly, Eqs. \eqref{p3,3-beau relation} and \eqref{p9,3-beau relation} resemble Hirschhorn's result for $p_{-3}(n)$ \cite[Eq. (6.3)]{Hir2017}:
\begin{align*}
\sum_{n=0}^{\infty}p_{-3}(3n)q^{n}\sum_{n=0}^{\infty}p_{-3}(3n+2)q^{n}=\left(\sum_{n=0}^{\infty}p_{-3}(3n+1)q^{n}\right)^{2},
\end{align*}
where the $k$-colored partition function is defined by
\begin{align*}
\sum_{n=0}^{\infty}p_{-k}(n)q^{n}=\dfrac{1}{(q;q)_{\infty}^{k}}.
\end{align*}

On the other hand, Saikia and Boruah \cite[Theorem 4.3]{SB2017} obtained an infinite family of congruences for $3$-regular partition triples $T_{3}(n)$ modulo 9, where the generating function of $T_{\ell}(n)$ is defined by
\begin{align*}
\sum_{n=0}^{\infty}T_{\ell}(n)q^{n}=\dfrac{(q^{\ell};q^{\ell})_{\infty}^{3}}{(q;q)_{\infty}^{3}}.
\end{align*}

Following the similar strategy and applying the operator $H$, we can also obtain several infinite families of congruences modulo powers of 3 for $T_{3}(n)$ and $T_{9}(n)$similar to \eqref{case 1:odd}--\eqref{case 3:3n+2}.

Finally, define the sequences $p_{k,\ell}(n)$ and $t_{k,\ell}(n)$ by
\begin{align}
\sum_{n=0}^{\infty}p_{k,\ell}(n)q^{n} &=\dfrac{1}{(q;q)_{\infty}^{\ell}(q^{k};q^{k})_{\infty}^{\ell}},\label{gf:1}\\
\sum_{n=0}^{\infty}t_{k,\ell}(n)q^{n} &=\dfrac{(q^{k};q^{k})_{\infty}^{\ell}}{(q;q)_{\infty}^{\ell}}.\label{gf:2}
\end{align}

Taking $k=3$ or $9$, and $\ell=3m$ in  \eqref{gf:1} and \eqref{gf:2}, where $m$ is a given positive integer. Following the similar strategy in this paper, we can also investigate arithmetic properties modulo powers of 3 for $p_{k,\ell}(n)$ and $t_{k,\ell}(n)$ case by case.

With the help of modular forms, Atkin \cite{Atk1968} obtained the following powerful theorem for $k$-colored partitions.

\begin{theorem}[Theorem 1.1, \cite{Atk1968}]\label{powerful thm}
For any $k>0$ and $q=2,3,5,7$ or $13$. If $24n\equiv k\pmod{q^{r}}$, then $p_{-k}(n)\equiv0\pmod{q^{\frac{1}{2}\alpha r+\epsilon}}$, where $\epsilon=\epsilon(q,k)=O(\log k)$, and where $\alpha$ depends on $q$ and the residue of $k$ modulo $24$ according to a certain table.
\end{theorem}

It would be appealing to find analogous results for $p_{k,\ell}(n)$ and $t_{k,\ell}(n)$ for arbitrary $\ell=3m$ ($m\in\mathbb{N}_{+}$) and $k=3$ or 9, which parallel to Atkin's results for $q=3$ in Theorem \ref{powerful thm}.

\section*{Acknowledgement}
I am indebted to Shishuo Fu for his helpful comments and suggestions that have improved this paper to a great extent. I would like to acknowledge the referee for his/her careful reading and helpful comments on an earlier version of the paper. This work was supported by the National Natural Science Foundation of China (No.~11501061).

\end{document}